\documentclass[11pt]{article}

\usepackage[utf8]{inputenc}
\usepackage{amsmath,amsthm}
\usepackage{amsfonts}
\usepackage{pdflscape}
\usepackage[usenames]{color}

\usepackage{authblk}

\usepackage{tikz-cd, tikz, graphicx}

\usepackage{subcaption}

\usetikzlibrary{patterns,snakes}

\newtheorem{theorem}{Theorem}
\newtheorem{proposition}[theorem]{Proposition}
\newtheorem{lemma}[theorem]{Lemma}

\newtheorem{corollary}[theorem]{Corollary}

\newenvironment{keywords}{
    \noindent\textbf{Keywords: }
    \itshape
}

\title{On the spectral radius of unbalanced signed bipartite graphs}

\author[1]{Cristian M. Conde\thanks{cconde@campus.ungs.edu.ar}}
\author[1]{Ezequiel Dratman \thanks{edratman@campus.ungs.edu.ar}}
\author[1]{Luciano N. Grippo \thanks{lgrippo@campus.ungs.edu.ar}}

\date{}

\affil[1]{Universidad Nacional de General Sarmiento. Instituto de Ciencias, Argentina. CONICET, ICI-UNGS, Buenos Aires, Argentina}

\begin{document}
	\maketitle

	\begin{abstract}
		A signed graph is one that features two types of edges: positive and negative. Balanced signed graphs are those in which all cycles contain an even number of positive edges. In the adjacency matrix of a signed graph, entries can be $0$, $-1$, or $1$, depending on whether $ij$ represents no edge, a negative edge, or a positive edge, respectively. The index of the adjacency matrix of a signed graph $\dot{G}$ is less or equal to the index of the adjacency matrix of its underlying graph $G$, i.e., $\lambda_1(\dot{G}) \le \lambda_1(G)$. Indeed, if $\dot{G}$ is balanced, then $\lambda_1(\dot{G})=\lambda_1(G)$. This inequality becomes strict when $\dot{G}$ is an unbalanced signed graph. Recently, Brunetti and Stani\'c found the whole list of unbalanced signed graphs on $n$ vertices with maximum (resp. minimum) spectral radius. To our knowledge, there has been little research on this problem when unbalanced signed graphs are confined to specific graph classes. In this article, we demonstrate that there is only one unbalanced signed bipartite graph on $n$ vertices with maximum spectral radius, up to an operation on the signed edges known as switching. Additionally, we investigate unbalanced signed complete bipartite graphs on $n$ vertices with a bounded number of edges and maximum spectral radius, where the negative edges induce a tree.  
	\end{abstract}
	
\begin{keywords}
signed bipartite graphs; spectral radius; unbalanced signed graphs
\end{keywords}

	\section{Introduction}
	
		Given a graph $G$, a \emph{signed graph} $\dot{G}$ is defined as a pair $(G,\sigma)$, where $\sigma$ is the \emph{signature} function that assigns a value $\sigma(ij)\in \{-1,1\}$ to each edge $ij\in E(G)$. For convenience, we denote $\dot{G}=(G,H^{-})$, where $H$ is the induced subgraph for the negative edges of  $\dot{G}$.  The edge set of $\dot{G}$ comprises both positive and negative edges. We interpret any given graph as a signed graph, where all edges are positive by default, and vice versa. The graph $G$ is referred to as the \emph{underlying} graph of $\dot{G}$, and it can be obtained by reversing the sign of each negative edge in $\dot{G}$. A signed graph $\dot G$ is termed \emph{balanced} if each of its cycles has an even number of negative edges. The adjacency matrix of a signed graph $\dot{G}$ is represented by a square matrix $A(\dot{G})=(a^{\sigma}_{ij})$, where $a^{\sigma}_{ij} =\sigma(ij)a_{ij}$, and $A(G)=(a_{ij})$ denotes the adjacency matrix of $G$. We denote by $\lambda_1(G)$ and $\rho(G)$ the maximum eigenvalue and the spectral radius of $A(G)$, respectively.  In addition, we refer to $\lambda_1(G)$ and $\rho(G)$ as the index and the spectral radius, respectively. 
	
		Given a signed graph $\dot{G}=(G, \sigma)$ and a subset $U\subset V(G)$, $\dot{G}^U$ denotes the signed graph derived from $\dot{G}$ by reversing the signs of the edges in $[U,V(G)\setminus U]=\{uv:\, u\in U\mbox{ and }v\in V(G)\setminus U\}$. This operation is commonly referred to as \emph{switching}, a concept introduced in \cite{Zaslavsky}. Two signed graphs $\dot F$ and $\dot G$ are \emph{switching equivalent} if $\dot F$ can be derived from $\dot G$ through a switching transformation. A signed graph is balanced if it is switching equivalent to its underlying graph \cite{Zaslavsky}. 
		
		The problem of finding the list of graphs with $n$ vertices within a specific graph class that have maximum spectral radius or index, has attracted significant interest among researchers, particularly those engaged in spectral graph theory. In the context of signed graphs, it is known that for the graphs $\dot G$ and $G$, the relation $\lambda_1(\dot G)\le\lambda_1(G)$ holds; equality is achieved if and only if $\dot G$  is balanced \cite{Stanic2019}. This fact prompts the search for the list of unbalanced signed graphs on $n$ vertices, with underlying graphs that are part of a certain graph class $\mathcal G$, and that have either the maximum index or spectral radius. 
		
		In 2017, Koledin and Stani\'c \cite{Koledinetal2017} presented a structural result regarding signed graphs with a fixed order, size, and a specific number of negative edges, all having maximum index. Furthermore, in their paper, they proposed candidates for achieving the maximum index in cases where the underlying graph is complete. In 2018, Stani\'c \cite{Stanic2018} examined the behavior of the index of signed graphs under minor perturbations. Subsequently, in 2020, Akbari et al. \cite{Akbarietal2020} identified signed complete graphs with $n$ vertices and $m$ negative edges inducing a tree, where $m<n-1$, that achieve the maximum index. Furthermore, Ghorbani and Majide \cite{Ghorbanietal2021} proved a revised version of the conjecture originally proposed by Koledin and Stani\'c in \cite{Koledinetal2017}. Li et al.\cite{Lietal2023} addressed the remaining case ($m=n-1$) by computing the unbalanced signed complete graph with the maximum index, where the negative edges induce a spanning tree. Further advances related to singed graphs can be found in \cite{DeqiongLi2023, Wang2023, haemers24, wang2024}.
 
		We define $-\dot{G}$ as the signed graph obtained from $\dot G$ by reversing the sign of each edge. We say $\dot G$ is \emph{sign-symmetric} if it is switching equivalent to $-\dot G$. If $\dot G$ is a sign-symmetric signed graph, then $\rho(\dot G)=\lambda_1(\dot{G})$. Note that if $G$ is a bipartite graph, then $\dot G$ is sign-symmetric. For further insights into sign-symmetric signed graphs, we direct interested readers to \cite{Belardo-2018}. 
		
		One aspect of extremal graph theory involves identifying graphs with specific extremal spectral parameters, such as the spectral radius or the index, within a defined class of graphs. In recent years, there has been substantial interest in exploring this direction, particularly in the field of signed graphs. Belardo et al. identified unbalanced signed unicyclic graphs on $n$ vertices with extremal spectral radius and index \cite{Belardo-2018}. More recently, Brunetti and Stani\'c have discovered unbalanced signed unicyclic graphs with extremal radii and indices \cite{BrunettiandStanic2022}. Our contribution to this field comprises two key findings, both centered on unbalanced signed bipartite graphs, as discussed in the preceding paragraph, which are sign-symmetric, and thus their indices are equal to their spectral radii. 
		
		Firstly, we identify unbalanced signed bipartite graphs whose underlying graph is isomorphic to $K_{r,s}$ with $r\leq s$, featuring at most $\frac{s}{2}-1$ negative edges inducing a tree, and having the maximum spectral radius. Secondly, we detect unbalanced signed bipartite graphs with $n$ vertices that achieve the maximum spectral radius.
		
		The article is organized as follows. Section \ref{sec: preliminaries} introduces some definitions and states preliminary results used throughout this article. In Section \ref{sec: Balanced signed complete bipartite graphs}, we characterize balanced signed complete bipartite graphs. In Section \ref{sec: Negative edges inducing a tree}, we list the unbalanced signed complete bipartite graphs with prescribed partition whose negative edges are bounded, induce a tree, and have maximum spectral radius. Section \ref{sec: Unbalanced signed complete bipartite graphs}  characterizes  unbalanced signed bipartite graph with a fixed number of vertices and maximum spectral radius. We conclude this manuscript with Section \ref{sec: Conclusion and further results}, which is devoted to conclusions and further results. 
	\section{Preliminaries}\label{sec: preliminaries}
	Let $\dot{G}=(G,H^{-})$ be a signed graph, it is noteworthy that $H$ is not empty graph and it does not have isolated vertices.	We use $\mathcal{M}_{r\times s}(\mathbb{R})$ to denote the collection of matrices with coefficients in $\mathbb R$, $r$ rows and $s$ columns, respectively, and in particular, when $r=s$, we employ $\mathcal{M}_r(\mathbb{R})$. 
	
	We will state a few results used throughout the article.
	\begin{lemma}\label{lem: adjacency matrix singned graphs}
		Let $\dot{G}=(G,H^{-})$ be a signed graph. Then we have
		\begin{equation*}
		A(\dot{G})=A(G)-2A(H). 
		\end{equation*}
	\end{lemma}	
%
\begin{lemma}
	If $C\in \mathcal{M}_n(\mathbb R)$, then $\rho(C)=\sqrt{ \rho(C^2)}$.
\end{lemma}
Given $A \in \mathcal{M}_n(\mathbb R)$, we use $\textrm{Spec}(A)$ to denote the multiset of eigenvalues of $A$, where each eigenvalue appears according to its algebraic multiplicity. 
\begin{lemma}
	If $A\in \mathcal{M}_{r\times s}(\mathbb R)$ and $B\in \mathcal{M}_{s\times r}(\mathbb R)$, then $\textrm{Spec}(AB)\setminus\{0\}=\textrm{Spec}(BA)\setminus\{0\}$. 
\end{lemma}
\begin{corollary}\label{cor: square spectrum}
	Let $A\in \mathcal{M}_{r\times s}(\mathbb R)$ and $B\in \mathcal{M}_{s\times r}(\mathbb R)$ be two matrices. If
	\[C=\begin{pmatrix}
		AB& 0\\
		0 & BA
		\end{pmatrix},\]
	then $\textrm{Spec}(C)\setminus\{0\}=\textrm{Spec}(AB)\setminus\{0\}$. In addition, $\lambda_1(C)=\lambda_1(AB)$.
\end{corollary}

Given a matrix $A \in \mathcal{M}_{r\times s}(\mathbb R)$, we say that $A\ge0$ if $a_{ij}\ge 0$ for every $1\le i\le r$ and $1\le j\le s$.

\begin{theorem}\label{thm: Perron}
    Let $A\in \mathcal{M}_n(\mathbb{R})$. If $A\ge 0$, then $\lambda_1(A)=\rho(A)$ and there exists a nonnegative vector $x$ such that $Ax=\lambda_1(A) x$. In addition, if $A>0$, then $\lambda_1(A)$ is a simple eigenvalue and there exists a positive vector $x$ such that $Ax=\lambda_1(A) x$.
\end{theorem}
Let $A$ be an $n\times n$ matrix partitioned as follows:
\[\left(\begin{matrix}
	A_{11}& \cdots& A_{1t}\\
	\vdots& \ddots& \vdots\\
	A_{t1}& \cdots& A_{tt}
\end{matrix}\right),
\]  
where $A_{ij} \in \mathcal{M}_{n_i\times n_j}(\mathbb{R})$ for every $1\le i, j\le t$, and $n_1+\cdots+n_t=n$. The \emph{quotient matrix} of $A$ under this partition is the matrix $B \in \mathcal{M}_{t\times t}(\mathbb{R})$ such that $b_{ij}=\frac 1 {n_i}\sum_{r=1}^{n_i}\sum_{s=1}^{n_j} a_{rs}^{ij}$, where $a_{rs}^{ij}$ stands for the $(r, s)$-entry of $A_{ij}$.
\begin{lemma}\cite{Lihuaetal2019}\label{lem: qotient matrix}
	Let $B$ be the equitable quotient matrix of $A$. Then $\textrm{Spec}(B)\subseteq \textrm{Spec}(A)$.
\end{lemma}
A graph $G$ is \emph{bipartite} if $V(G)$ can be partitioned into two independent sets $X$ and $Y$; that is, sets whose vertices are pairwise nonadjacent. The pair $\{X,\,Y\}$ is called a \emph{bipartition} of $G$. A \emph{complete bipartite} graph is a bipartite graph $G$ with a bipartition $\{X,\,Y\}$ such that $uv\in E(G)$ for each $u\in X$ and $v\in Y$. We use $K_{r,s}$ to denote the complete bipartite graph with a bipartition set with $r$ and $s$ vertices, respectively. A \emph{complete graph} is a graph in which every pair of vertices is adjacent. We use $K_n$ to denote the complete graph on $n$ vertices. Given two graphs $G$ and $H$, we denote by $G+H$ the disjoint union of $G$ and $H$.
\begin{lemma}\label{lem: compelte bipartite}
	Let $G$ be a connected bipartite graph. Then, $G$ is a complete bipartite graph if and only if $G$ does not contain $K_2+K_1$ as induced subgraph.
\end{lemma}
Let $\dot G=(G,\sigma)$ be a signed graph, and let $\theta: V(G)\to \{-1,1\}$ be a function. We define $\dot G^\theta=(G,\sigma^\theta)$ as the signed graph whose signature $\sigma^\theta$ is defined in terms of $\theta$ as follows: $\sigma^\theta(uv)=\theta(u)\sigma(uv)\theta(v)$, for each edge $uv\in E(G)$. Note that $\dot{G}^\theta=\dot{G}^U$ with $U=\{v \in V(G): \theta(v)=1\}$. Two signed graphs $\dot{G}_1=(G,\sigma_1)$ and $\dot{G}_2=(G,\sigma_2)$ are \emph{switching equivalents} if there exists a function $\theta: V(G)\to \{-1,1\}$ such that 
$\sigma_2(uv)=\sigma_1^\theta(uv)$ for each $uv\in E(G)$. Notice that switching equivalent graphs have the same spectrum. The following lemma is well known \cite{Zaslavsky}.
\begin{lemma}\label{lem: switching equivalent} 
	A signed graph $\dot G$ is balanced if and only if $\dot G$ is switching equivalent to its underlying graph.
\end{lemma}
\begin{lemma}\cite[Lemma 1]{Stanic2018}\label{lem: spectral radius switching equivalent class}  
	Let $\mathcal E$ denote a class of switching equivalent signed graphs and let $\lambda$ be an eigenvalue belonging to the common spectrum. Then $\mathcal E$ contains a signed graph for which the eigenvector that corresponds to $\lambda$ may be chosen in such a way that all its non-zero coordinates are of the same sign.
\end{lemma}
%
The following lemma we will be used several times to compare eigenvalues of two graphs.
\begin{lemma}\label{lem: comparing polynomials}
	Let $\dot F$ and $\dot G$ be two signed graphs, such that $P_{\dot F}(x)>P_{\dot G}(x)$ for every $x\ge\lambda_1(\dot G)$, where $P_{\dot F}$ denotes the characteristic polynomial of $A(\dot{F})$. Then, $\lambda_1(\dot F)<\lambda_1(\dot G)$.
\end{lemma}


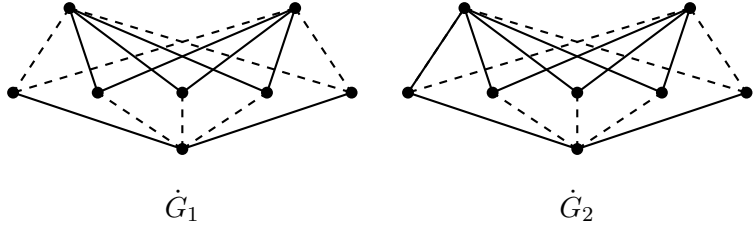
\begin{figure}
	\begin{center}
	
	\begin{tikzpicture}[scale=0.75]
	
	\coordinate (v1) at (-2,1.5);
	\coordinate (u1) at (0,-1);
	\coordinate (v2) at (2,1.5);
	
	\coordinate (u2) at (-3,0); 
	\coordinate (v3) at (-1.5,0);
	\coordinate (v4) at (0,0);
	\coordinate (v5) at (1.5,0);
	\coordinate (u3) at (3,0);
		
	\foreach \punto in {1,...,3}
        \fill (u\punto) circle(3pt);
	
	\foreach \punto in {1,...,5}
        \fill (v\punto) circle(3pt);
	
	\foreach \punto in {1,...,2}
        \foreach \qunto in {2,...,3}
            \draw[thick,dashed] (v\punto)--(u\qunto);
            
    \foreach \punto in {1,...,1}
        \foreach \qunto in {3,...,5}
            \draw[thick, dashed] (u\punto)--(v\qunto);
	
	\foreach \punto in {1,...,2}
        \foreach \qunto in {3,...,5}
            \draw[thick] (v\punto)--(v\qunto);
            
    \foreach \punto in {1,...,1}
        \foreach \qunto in {2,...,3}
            \draw[thick] (u\punto)--(u\qunto);
		
	\node[] at (0,-2) {$\dot{G}_1$};
	
    
    \coordinate (v1) at (5,1.5);
	\coordinate (u1) at (7,-1);
	\coordinate (v2) at (9,1.5);
	
	\coordinate (u2) at (4,0); 
	\coordinate (v3) at (5.5,0);
	\coordinate (v4) at (7,0);
	\coordinate (v5) at (8.5,0);
	\coordinate (u3) at (10,0);
		
	\foreach \punto in {1,...,3}
        \fill (u\punto) circle(3pt);
	
	\foreach \punto in {1,...,5}
        \fill (v\punto) circle(3pt);
	
	\foreach \punto in {1,...,2}
        \foreach \qunto in {2,...,3}
            \draw[thick,dashed] (v\punto)--(u\qunto);
    
    \draw[thick] (v1)--(u2);
    
    \foreach \punto in {1,...,1}
        \foreach \qunto in {3,...,5}
            \draw[thick, dashed] (u\punto)--(v\qunto);
	
	\foreach \punto in {1,...,2}
        \foreach \qunto in {3,...,5}
            \draw[thick] (v\punto)--(v\qunto);
            
    \foreach \punto in {1,...,1}
        \foreach \qunto in {2,...,3}
            \draw[thick] (u\punto)--(u\qunto);

	\node[] at (7,-2) {$\dot{G}_2$};
	\end{tikzpicture}
	
	\end{center}
	
	\caption{$\dot{G}_1$ and $\dot{G}_2$ are balanced and unbalanced signed complete bipartite graphs, respectively}\label{fig: Balanceado y No balanceado}
\end{figure}

\section{Balanced signed complete bipartite graphs}\label{sec: Balanced signed complete bipartite graphs}
In this section, we characterize the balanced signed complete bipartite graphs (see Fig. \ref{fig: Balanceado y No balanceado}).
\begin{theorem}\label{Thm: balanced signed complete bipartite}
	Let $\dot{G}=(K_{r,s},H^-)$ be a signed complete bipartite graph with $r\le s$, and let $\{X,Y\}$ be a bipartition with $|X|=r$ and $|Y|=s$. Then, $\dot{G}$ is balanced if and only if one of the following conditions holds:
	\begin{enumerate}
		\item\label{item1} $H$ is a complete bipartite graph with $r<|V(H)|<r+s$, and a bipartition $\{X_H,Y_H\}$ such that $X_H\subseteq X$, $Y_H\subseteq Y$, and $|X_H|=r$ or $|Y_H|=s$.
		\item\label{item2}  $H$ either is a complete bipartite graph or the disjoint union of two complete bipartite graphs, with $|V(H)|=r+s$.
	\end{enumerate}
\end{theorem}
\begin{proof}
	Notice that if $r=1$, the result trivially holds. Therefore, we can assume that $2\le r$. 
	
	Let us prove the ``only if part''. Assume that $\dot{G}$ is a balanced signed graph. Suppose, towards a contradiction, that $|V(H)|\le r$. Hence,  $X\setminus X_H\neq\emptyset$ and $Y\setminus Y_H\neq\emptyset$. Consequently, there exists a four-vertex set $\{u,v,x,y\}$, where $uv\in E(H)$ (is a negative edge), $x\in X\setminus X_H$ and  $y\in Y\setminus Y_H$, inducing a $C_4$ with one negative edge, contradicting the balancedness of $\dot{G}$. Thus, $r<|V(H)|$. Notice that, we have already proved that if $\dot{G}$ is balanced, then either $X_H=X$ or $Y_H=Y$.  
	
	First, assume that $|V(H)|<r+s$. Suppose, towards a contradiction, that $H$ is not a complete bipartite graph. By Lemma \ref{lem: compelte bipartite}, there exists a three-vertex set ${a,b,c}$ that induces the graph $K_2+K_1$. Suppose, without losing generality, that $ab\in E(H)$, $a,c\in X_H$ and $b\in Y_H$. If $c$ and $b$ are in the same connected component of $H$ and thus $a$ is also in the same connected component of $H$ than $c$, there exists a path $P_{cb}$ of negative edges linking $c$ and $b$. Hence, the edges of $P_{cb}$ plus the positive edge $bc$ is a cycle $C$ with exactly one positive edge. Since any cycle in $K_{r,s}$ has even length, $C$ has an odd number of negative edges, contradicting the balancedness of $\dot{G}$. Thus $\{a,b\}$ and $c$ are in distinct connected component of $H$. 
	
	Suppose that $X_H=X$. Since $|V(H)|<r+s$ there exists a vertex $d\in Y$ such that $ad,\,cd\notin E(H)$. Consequently, $\{a,b,c,d\}$ induces a $C_4$ with exactly one negative edge, contradicting the balanceness of $\dot{G}$. Suppose now $Y_H=Y$. Since $H$ has no isolated vertices, there exists a vertex $e\in Y_H$ such that $ce\in E(H)$. Since $a$ and $c$ are in distinct connected components of $H$, $a$ and $e$ are also in distinct connected components of $H$. Since $|V(H)|<r+s$, there exits a vertex $f\in X$ such that $fb,\,fe\notin E(H)$. Consequently, $\{b,c,e,f\}$ induces a $C_4$ with exactly one negative edge, contradicting the balanceness of $\dot{G}$. In all cases we have reached a contradiction coming from assuming that $H$ has $K_2+K_1$ as induced subgraph. Therefore, $H$ is a complete bipartite graph. 
	
	Assume now that $|V(H)|=r+s$. Analogously to the case $|V(H)<r+s$ we can prove that each connected component of $H$ is a complete bipartite graph. Suppose, towards a contradiction, that $H$ has at least three connected components. Let $H_1,\,H_2$ and $H_3$, three connected components of $H$. Hence, there exists a negative edge $rs\in E(H_1)$, a vertex $t\in V(H_2)\cap X$, and a vertex $u\in V(H_3)\cap Y$. Consequently, $\{r,s,t,u\}$ induces a $C_4$ with exactly one negative edge, contradicting the balanceness of $\dot{G}$. The contradiction arose from assuming that $H$ has more than three connected components. Therefore, $H$ is the disjoint union of at most two bipartite graphs. 
	
	Let us prove the ``if part''. Assume that $\dot{G}$ is a singed complete bipartite graph as described in \ref{item1}. Assume that $|X_H|=r$. Consider the function $\theta: V(G)\to\{-1,1\}$ defined as follows,
	\[\theta(v)=\begin{cases}
		1& \text{if } v\in X\cup (Y\setminus Y_H),\\
		-1& \text{if } v\in  Y_H.
	\end{cases}
	\]
	Thus, $\dot{G}^\theta$ has all its edges positive. By Lemma \ref{lem: switching equivalent}, it follows that $\dot{G}$ is a balanced signed graph. We leave the case $|Y_H|=s$ for the reader, as it is symmetric to the one we have considered. Now, assume that $\dot{G}$ is a complete signed bipartite graph as in \ref{item2}. Notice that if all its edges are negative, clearly $\dot{G}$ is a balanced signed graph. Suppose that $H$ is the disjoint union of two complete bipartite graphs $H_1$ and $H_2$, with bipartitions $\{X_1,Y_1\}$ and $\{X_2,Y_2\}$, respectively, where $X_1\cup X_2=X$ and $Y_1\cup Y_2=Y$. Consider the function $\theta: V(G)\to\{-1,1\}$ defined as follows,
	\[\theta(v)=\begin{cases}
		1& \text{if } v\in X_1\cup Y_2,\\
		-1& \text{if } v\in  X_2\cup Y_1.
	\end{cases}
	\]
	Thus $\dot{G}^\theta$ has all its edges positive. By Lemma \ref{lem: switching equivalent}, it follows that $\dot{G}$ is a balanced signed graph.
\end{proof}
\section{Negative edges inducing a tree}\label{sec: Negative edges inducing a tree}
In this section, we will consider a signed complete bipartite graph $\dot{G}=(K_{r,s},H_m^-)$ with a bipartition $\{X,Y\}$ such that $X=\{v_1,\ldots,v_r\}$, $Y=\{w_1,\ldots,w_s\}$, and $r\le s$. Here, $H_m$ denotes the subgraph induced by the $m$ negative edges of $\dot{G}$. A \emph{chain graph} is a bipartite graph with a bipartition $\{U,\, W\}$ such that for any $u,z\in U$, either $N(u)\subseteq N(z)$ or $N(z)\subseteq N(u)$ (see Fig. \ref{fig: Chain}).

\begin{figure}
\begin{center}
\begin{tikzpicture}[scale=1]
    
    \coordinate (v1) at (0, 1);
    \coordinate (v2) at (2, 1);
    \coordinate (v3) at (4, 1);
    \coordinate (v4) at (6, 1);
    \coordinate (v5) at (8, 1);
   
    \coordinate (w1) at (1, -1);
    \coordinate (w2) at (3, -1);
    \coordinate (w3) at (5, -1);
    \coordinate (w4) at (7, -1);
    
    \foreach \punto in {1,...,5}
        \fill (v\punto) circle(3pt);
	
	\foreach \punto in {1,...,4}
        \fill (w\punto) circle(3pt);

    \node[above] at (v1) {$v_1$};
    \node[above] at (v2) {$v_2$};
    \node[above] at (v3) {$v_3$};
    \node[above] at (v4) {$v_4$};
    \node[above] at (v5) {$v_5$};
    
    \node[below] at (w1) {$w_1$};
    \node[below] at (w2) {$w_2$};
    \node[below] at (w3) {$w_3$};
    \node[below] at (w4) {$w_4$};

    \draw[thick] (v1) -- (w1);
    \draw[thick] (v2) -- (w1);
    \draw[thick] (v2) -- (w2);
    \foreach \i in {1,...,3} {
        \draw[thick] (v3) -- (w\i);
    }
   
    \foreach \i in {1,...,4} {
        \draw[thick] (v4) -- (w\i);
        \draw[thick] (v5) -- (w\i);
    }
   
\end{tikzpicture}
\end{center}

\caption{Chain graph}\label{fig: Chain}

\end{figure}
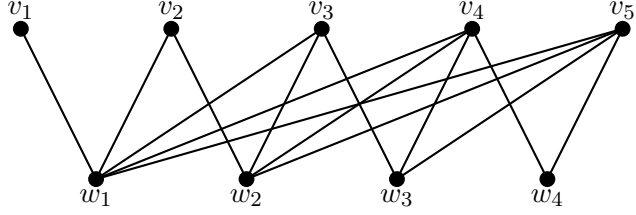

\begin{proposition}\label{thm: chain}
	Let $\mathcal K_{r,s,m}$  be the family of signed complete bipartite graphs $\dot{G}=(K_{r,s},H_m^-)$ such that $H_m$ satisfies the following conditions for each $1\le i<j\le r$ 
	\begin{enumerate}
		\item[a)]\label{item: thm of chain 1}	$d_{H_m}(v_i)+d_{H_m}(v_j)\le \frac s 2$, or 
		\item[b)]\label{item: thm of chain 2}	$\frac {3s} 2\le d_{H_m}(v_i)+d_{H_m}(v_j)$,
	\end{enumerate}
	for any $v_i, v_j \in V(H_m)$.
	
	If $(K_{r,s},{H^*_m}^-)$ maximizes the spectral radius among all signed graphs in $\mathcal K_{r,s,m}$, then the maximum spectral radius is achieved by at least one signed graph $\dot{G}=(K_{r,s},H_m^-)$ where $H_m$ is a chain graph. In addition, if the inequalities $a)$ and $b)$ are strict, then $H^*_m$ is a chain graph. 
\end{proposition}
\begin{proof}
	Let $\dot{G} \in \mathcal K_{r,s,m}$. By Lemma \ref{lem: adjacency matrix singned graphs}, it follows that
	\[A(\dot{G})^2=A(K_{r,s})^2-2A(K_{r,s})A(H_m)-2A(H_m)A(K_{r,s})+4A(H_m)^2,\]
	where the rows of $A(K_{r,s})$ are indexed by $v_1,\ldots,v_r,w_1,\ldots,w_s$. In addition,
	\begin{equation}\label{ec: Cuadrado de A con Br y Bs}
	A(\dot{G})^2=\begin{pmatrix}
		B & 0\\
		0   & C
	\end{pmatrix},
 	\end{equation}
	where 
	\[B_{i,j}=s-2(d_H(v_j)+d_H(v_i))+4|N_H(v_i)\cap N_H(v_j)|,\]
	for any $1\le i, j \le r$.
	
	Since $\dot{G}$ is a signed bipartite graph, it is easy to see that
	\[A(\dot{G})^2=\begin{pmatrix}
		M M^t & 0\\
		0   & M^t M
	\end{pmatrix},\]
	for some $M\in \mathcal{M}_{r\times s}(\mathbb{R})$. Combining this equality, \eqref{ec: Cuadrado de A con Br y Bs} and Corollary \ref{cor: square spectrum}, we conclude that $\lambda_1^2(\dot{G})= \lambda_1(B)$.   
	
	On the one hand, if $d_H(v_i)+d_H(v_j)\le\frac s 2$, then $B_{i,j}\ge s-2(d_H(v_i)+d_H(v_j))\ge 0$. On the other hand, if $d_H(v_i)+d_H(v_j)\ge\frac {3s} 2$, then $|N_H(v_i)\cap N_H(v_j)|\ge \frac s 2$ and thus $B_{i,j}\ge 3s-2(d_H(v_i)+d_H(v_j))\ge 0$. In conclusion, $B\ge0$, and from Theorem \ref{thm: Perron}, we have that $\lambda_1(B)=\rho(B)$ and there exists a nonnegative vector $x_B$ such that $B x_B= \rho(B) x_B$. 
	
	Let $\dot{G}^*= (K_{r,s},{H^*_m}^-)$ be the signed graph which maximizes the spectral radius among all signed graphs in $\mathcal K_{r,s,m}$. Suppose $H_m^*$ is not a chain graph, without losing generality, we denote by $\{v_1,\ldots,v_k\}=N_{H_m^*}(w_2) \setminus N_{H_m^*}(w_1)$ and $\{v_{k+1},\ldots,v_\ell\}=N_{H_m^*}(w_1) \setminus N_{H_m^*}(w_2)$ with $\ell>k\ge 1 $. Let $\dot{G_1}=(K_{r,s},F_m^{-})$ be the signed graph such that $V(F_m)=V(H_m^*)$ and $E(F_m)=(E(H^*_m)\setminus\{v_iw_2 : 1\le i\le k \})\cup\{v_iw_1 :1\le i \le k \}$.  Consequently, 
	\[
	B_1-B^*=4 \sum_{i=1}^k \sum_{j=k+1}^{\ell} (e_ie_j^t + e_je_i^t),
	\]
	where $B_1$ and $B^*$ are the left upper block of $A(\dot{G_1})^2$ and $A(\dot{G}^*)^2$, respectively, as given in \eqref{ec: Cuadrado de A con Br y Bs}. 
	
	Let $\lambda_1^2(\dot{G}^*)=\rho(B^*)$ be the index of $A(\dot{G}^*)^2$, and let $x_{B^*}$ be a vector such that $B^*x_{B^*}=\lambda_1^2(\dot{G}^*)x_{B^*}$. Thus
	\[
	x_{B^*}^t \left(B_1-B^*\right)x_{B^*}\ge 0.
	\]
	Therefore, 
	\[
	\rho\left(A(\dot{G}_1)^2\right) = \rho(B_1) = \max_{x\in\mathbb{R}^r,\, x\neq \mathbf{0}}\frac{x^t B_1 x}{x^tx}\ge\rho(B^*) = \rho\left(A(\dot{G}^*)^2\right),
	\]
	Moreover, by Theorem \ref{thm: Perron}, if conditions $a)$ and $b)$ are strict, then the last inequality is also strict.
\end{proof}
Let $\mathcal {KT}_{r,s,m}$  be the family of complete bipartite signed graphs $\dot{G}=(K_{r,s},T_m^-)$ such that $r\le s$ and $T_m$ is a tree with $m$ edges. Note that if $m<r$, the signed graphs in $\mathcal {KT}_{r,s,m}$ are unbalanced due to Theorem \ref{Thm: balanced signed complete bipartite}; otherwise if $m=r$ and $T_m$ is the star $K_{1,m}$, then $\dot{G}$ is a balanced signed graph. Let us denote by $D_{i, j}$ the star or double star in $K_{r,s}$ with $V(D_{i,j})=\{v_1, \dots, v_i, w_1, \dots ,w_j\}$ and its centers are $v_1$ and $w_1$, where $\{X,\, Y\}$ is the bipartition of $K_{r,s}$ such that $X=\{v_1, \dots, v_r\}$ and $Y=\{w_1, \dots, w_s\}$ (see Fig. \ref{fig: doble estrella}). 

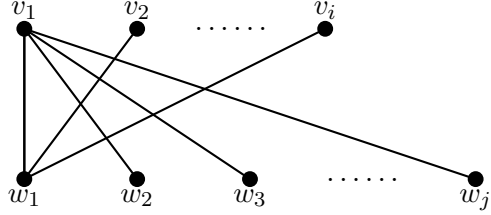
\begin{figure}

\begin{center}
\begin{tikzpicture}[scale=1]
    
    \coordinate (v1) at (0, 1);
    \coordinate (v2) at (1.5, 1);
    \coordinate (v3) at (4, 1);
    
    \coordinate  (w1) at (0, -1);
    \coordinate  (w2) at (1.5, -1);
    \coordinate  (w3) at (3, -1);
    \coordinate  (w4) at (6, -1);
    
    \foreach \punto in {1,...,3}
        \fill (v\punto) circle(3pt);
	
	\foreach \punto in {1,...,4}
        \fill (w\punto) circle(3pt);
	
    \node[above] at (v1) {$v_1$};
    \node[above] at (v2) {$v_2$};
    \node[above] at (v3) {$v_i$};
    
    \node[below] at (w1) {$w_1$};
    \node[below] at (w2) {$w_2$};
    \node[below] at (w3) {$w_3$};
    \node[below] at (w4) {$w_j$};

    \draw[thick] (v1) -- (w1);
    \draw[thick] (v1) -- (w2);
    \draw[thick] (v1) -- (w3);
    \draw[thick] (v1) -- (w4);
    
    \draw[thick] (w1) -- (v1);
    \draw[thick] (w1) -- (v2);
    \draw[thick] (w1) -- (v3);
   
    \node at (2.75,1) {$\ldots\ldots$};
    \node at (4.5,-1) {$\ldots\ldots$};   
    
\end{tikzpicture}
\end{center}

\caption{Double star $D_{i,j}$}\label{fig: doble estrella} 

\end{figure}

We will continue by presenting two technical results.
\begin{lemma}\label{lem: pol stars and double stars}
		Let $\mathcal {KDS}_{r,s,k}$  be the family of complete bipartite signed graphs $\dot{G}=(K_{r,s},H_k^-)$ such that $H_k$ is a star or double star with $k$ edges. If  $P$ denotes the characteristic polynomial of $B$, where $B$ is the left upper block of $A(\dot{G})^2$ given in \eqref{ec: Cuadrado de A con Br y Bs}, then
		\[P(\lambda)=\lambda^3-rs\lambda^2+Q_1\lambda +Q_0,\]
		where \[Q_1=4(d-1)(s(k-d)+(sk-dr))+4(r-k)(s-1)k,\]\[Q_0=16(d-s)(d-1)(k-d)(r-k+d-1),\]
		and  $d=d_{H_k}(v_1)$.
	\end{lemma}
	\begin{proof}
		By mimicking the proof of Proposition \ref{thm: chain}, we have that 
		$$
		B =s J_{r}-2
			\begin{pmatrix}
				0 & x^t &  y^t\\
				x &  0 & J_{l-1, r-l}\\
				y &  J_{r-l, l-1}  & 0 
			\end{pmatrix},
		$$
		where $l=k-d+1=d_{H_k}(w_1)$, $x=(d-1)J_{l-1,1}$ and $y=dJ_{r-l,1}$. \\
		After some elementary row and column operations on the matrix $B$, we obtain that such matrix is similar to 
		\[\begin{pmatrix}
			\widehat{B} & C\\
			0 & 0
		\end{pmatrix},\]
		where \[\widehat{B}=\begin{pmatrix}
			s&(l-1)(s-2(d-1))& (r-l)(s-2d)\\
			(s-2(d-1))&(l-1)s& (r-l)(s-2)\\
			(s-2d)&(l-1)(s-2)&s(r-l)
		\end{pmatrix}.\]
		Then, we have 
		$$
			P(\lambda)=\det(\widehat{B}-\lambda I)=\lambda^3-rs\lambda^2+Q_1\lambda +Q_0,
        $$
		with $Q_1=4(d-1)(s(k-d)+(sk-dr))+4(r-k)(s-1)k$ and $Q_0=16(d-s)(d-1)(k-d)(r-k+d-1)$.
\end{proof}
Now, building upon the previous lemma, we will derive the spectrum of $\dot{G}=(K_{r,s},H_k^-)$, where $H_k$ represents either a star or a double star, for specific values of the parameters $d$ and $l$. 
\begin{corollary}\label{cor: polinomios de estrellas}
Let $\dot{G}=(K_{r,s},H_k^-)$ be a complete bipartite signed graphs such that $H_k$ is a star or double star with $k$ edges. If $P$ is defined as in Lemma \ref{lem: pol stars and double stars}, then
	\begin{enumerate}
		\item  If  $H_k=D_{k,1}$, then 
		$$
			P(\lambda)=P^{(1)}(\lambda)= \lambda^3-rs\lambda^2+4(r-k)(s-1)k\lambda,
		$$
		and 
		\begin{small}
		$$
		Spec(A(\dot{G}))=\left\{\sqrt{\frac{rs-\sqrt{z}}{2}}, 0^{[r-2]}, \sqrt{\frac{rs+\sqrt{z}}{2}}\right\},
		$$
		\end{small}
		where $z=(rs)^2-16(r-k)(s-1)k$.
		\item  If  $H_k=D_{r,k+1-r}$, then 
		$$ 
		P(\lambda) = P^{(2)}(\lambda)= \lambda^3-rs\lambda^2+4(k-r)(r-1)(s+r-k)\lambda, 
		$$
		and 
		\begin{small}
		$$
		Spec(A(\dot{G}))=\left\{\sqrt{\frac{rs-\sqrt{z}}{2}}, 0^{[r-2]}, \sqrt{\frac{rs+\sqrt{z}}{2}}\right\},
		$$
		\end{small}
		where $z=(rs)^2-16(k-r)(r-1)(s+r-k)$.
		\item If  $H_k=D_{1,k}$, then  
		$$
			P(\lambda) = P^{(3)}(\lambda) = \lambda^3-rs\lambda^2+4k(r-1)(s-k)\lambda
		$$
		and 
		\begin{small}
		$$
		Spec(A(\dot{G}))=\left\{\sqrt{\frac{rs-\sqrt{z}}{2}}, 0^{[r-2]}, \sqrt{\frac{rs+\sqrt{z}}{2}}\right\},
		$$
		\end{small}
		where $z=(rs)^2-16k(r-1)(s-k)$.
		\item  If  $H_k=D_{k+1-s,s}$, then 
		$$
			P(\lambda)=P^{(4)}(\lambda)=\lambda^3-rs\lambda^2 +4(r-k+s)(s-1)(k-s)\lambda,
		$$
		and 
		\begin{small}
		$$
		Spec(A(\dot{G}))=\left\{\sqrt{\frac{rs-\sqrt{z}}{2}}, 0^{[r-2]}, \sqrt{\frac{rs+\sqrt{z}}{2}}\right\},
		$$
		\end{small}
		where $z=(rs)^2-16(r-k+s)(s-1)(k-s)$.
\end{enumerate}
 
\end{corollary}
The following lemma will be used to prove the main result of this section, Theorem \ref{thm: complete bipartite graphs negative edges induce a tree}.
\begin{lemma}\label{thm: polynomials}
	Let $\mathcal {KDS}_{r,s,k}$  be the family of complete bipartite signed graphs $\dot{G}=(K_{r,s},H_k^-)$ such that $H_k$ is a star or double star with $k$ edges. Then, the maximum spectral radius in $\mathcal {KDS}_{r,s,k}$ is attained by:
	\begin{enumerate}
		\item $H_k=D_{k,1}$ when $k\le\frac{r+s}2$ and $k\le r$.
		\item $H_k=D_{r,k+1-r}$ when $k\le \frac{r+s} 2$ and $k> r$.
		\item $H_k=D_{1,k}$ when $k\ge \frac{r+s} 2$ and $k< s$.
		\item $H_k=D_{k+1-s, s}$ when $k\ge \frac{r+s} 2$ and $k\ge s$.
	\end{enumerate}
\end{lemma}
\begin{proof}
	
	Let $\dot{G}=(K_{r,s},H_k^-)$ be a complete bipartite signed graph where $H_k$ is a star or double star with $k$ edges. Let $B$, $P$, $d$ and $l$ be defined as in Lemma \ref{lem: pol stars and double stars}. We have that $\rho(\dot{G}) = \sqrt{\rho(B)}$. 
	
	As consequence of Lemma \ref{lem: comparing polynomials}, the central idea of this proof is to compare $P$ with appropriate polynomials that have been previously examined and scrutinized in Corollary \ref{cor: polinomios de estrellas}, based on the varying values taken on by the parameters $d$ and $l$. We will explore the following cases.
	
	\begin{enumerate}
	 \item	If we assume that $k\leq \frac{r+s}{2}$ and $k\leq r$, then $k\geq d$. For $d=1$, or $d=k$ and $r=s$, the polynomial $P$ simplifies to $P^{(1)}$. We can further assume that $d\geq 2$, and $d<k$ or $r<s$. "Next, we examine the polynomial difference
	\begin{align*}
		P(\lambda)-P^{(1)}(\lambda)&=4(d-1)Q(\lambda),
	\end{align*}
	where \begin{align*}
		Q(\lambda)&=(s(k-d)+(sk-dr))\lambda-4(s-d)(k-d)(r-k+d-1)\\
		&=(2sk-d(s+r))\lambda-4(s-d)(k-d)(r-k+d-1).
	\end{align*}
	Notes that $Q(\lambda)$ represents is a line with a positive slope, as $2sk-d(s+r)=2s(k-d)+d(s-r)>0$. Since $\tilde{\lambda}=\frac{rs}{2}$ is not greater than the maximum positive root of $P^{(1)}$, it is sufficient to prove that $Q(\tilde{\lambda})>0$. Therefore, for $s\geq 4$, it follows that
	\begin{align*}
		Q(\tilde{\lambda})&=d(s-r)\tilde{\lambda}+s^2r(k-d)-4(s-d)(k-d)(r-k+d-1)\\
		&=d(s-r)\tilde{\lambda}+(k-d)(s^2r-4(s-d)(r-k+d-1))\\
		&\geq d(s-r)\tilde{\lambda}+(k-d)(s^2r-4(s-1)(r-1))>0.
	\end{align*}

	On the other hand, for $s<4$, the inequalities $4>s>r\geq k \geq d \geq 2$ or $4>s\ge r\geq k > d \geq 2$ lead to specific cases: $s=3, r=k=d=2$ or $s=r=k=3, d=2$. In these cases, $Q(\tilde{\lambda})>0.$
	
	\item If we assume that $k\leq \frac{r+s}{2}$ and $k > r$, then $s>\frac{s+r}{2}\geq k >r$. For $d=k-r+1$, or $d=k=\frac{r+s}{2}$, we obtain that $P=P^{(2)}$. So we can assume that $r\geq 2$. We will now analyze the polynomial
	\begin{align*}
		P(\lambda)-P^{(2)}(\lambda)&=4(d-k+r-1)Q(\lambda),
	\end{align*}
	where \begin{align*}
		Q(\lambda)&=[(k-d)(s-r)+(s+r-2d)r]\lambda-4(s-d)(d-1)(k-d).
	\end{align*}	
	We begin by considering the following specific cases:\\
	When $k=d$ and $k\neq \frac{r+s}{2}$, it follows that
	\begin{align*}
		P(\lambda)-P^{(2)}(\lambda)&=4(r-1)(s+r-2d)r\lambda>0, 
	\end{align*}
	for any $\lambda >0$ since $2d=2k<r+s$. \\
	Now, let us assume that $k>d$ and $k=\frac{r+s}{2}$. From the first inequality, we conclude that $(k-d)+(r-1)>0$, or equivalently $d>k-r+1$. By combining the inequalities $k>d>k-r+1$, we deduce that $r\geq 3$. Then, we have
	$$
	P(\lambda)-P^{(2)}(\lambda)
		\geq 4(d-k+r-1)\left(\frac{s+r}{2}-d\right)\left[(s+r)\lambda-(s-1)^2\right].
    $$
	Thus, we conclude that 
    $$
    P(\lambda)-P^{(2)}(\lambda) >0,
    $$
    for $\lambda \ge \tilde{\lambda}=\frac{rs}{2}$.
	
	Let $k>d>k-r+1$ and $k<\frac{r+s}{2}$, so we get 
	\begin{align*}
		Q(\lambda)&=((k-d)(s-r)+(s+r-2d)r)\lambda
		-4(s-d)(d-1)(k-d)\\
		&=(k-d)[s\lambda-4(s-1)(d-1)]+(s+r-k-d).\end{align*}
	In conclusion, since $Q(\tilde{\lambda}) \geq 0$, so $P(\lambda)> P^{(2)}(\lambda)$ for any $\lambda \geq \tilde{\lambda}$.
	
	\item If $k\geq \frac{r+s}{2}$ and $k<s$, then $r<k<s$ and $k\geq d\geq k-r+1$. In this case, we consider 
	\begin{align*}
		P(\lambda)-P^{(3)}(\lambda)&=4(k-d)Q(\lambda),
	\end{align*}
	where 
	$$
	Q(\lambda)=a_1\lambda +a_0,
	$$ 
	with 
	$$a_0=-4(s-d)(d-1)(r-k+d-1)],$$  
	and
	$$a_1=2(d-1)s-d(s-r)-(k-1)(s-r).$$ 
	Note that the linear coefficient of the polynomial $Q(\lambda)$ can be expressed as $a_1=(2k-(s+r))r+(s+r)(d-k+r-1)\geq 0$. Consequently, $P(\lambda)-P^{(3)}(\lambda)$ is a line with a positive slope.
	
	We consider the following specific cases\\
	If $d=k-r+1$ and $k=\frac{r+s}{2}$, or if $k=d$, then $P=P^{(3)}$. 
	
	Let  $k> d=k-r+1$ and $k>\frac{r+s}{2}$, then  \[P(\lambda)-P^{(3)}(\lambda)=4(k-d)(2k-(s+r))r\lambda>0, \]
	for any $\lambda >0.$
	
	If  $k>d>k-r+1$ and $k=\frac{r+s}{2}$, then  
	\[P(\lambda)-P^{(3)}(\lambda)=4(k-d)(d-k+r-1)[(s+r)\lambda -4(s-d)(d-1)]>0, \]
	for any $\lambda \geq \tilde{\lambda}=\frac{rs}{2}, $ indeed
	$$
	P(\tilde{\lambda})-P^{(3)}(\tilde{\lambda})\geq 4(k-d)(d-k+r-1)\left[(s+r)\frac{sr}{2} -(s-1)^2\right]>0.
	$$
	
	Finally, if $k>d>k-r+1$ and $k>\frac{r+s}{2}$, then 
	\begin{align*}
		Q(\lambda)&=(2k-(s+r))r\lambda +(d-k+r-1)[(s+r)\lambda-4(s-d)(d-1)]
	\end{align*}
	has positive slope,	and we have
	$$
	Q(\tilde{\lambda})>(d-k+r-1)[(s+r)s-(s-1)^2]>0.
	$$
	Therefore, $P(\lambda)-P^{(3)}(\lambda)>0$ if $\lambda \geq \tilde{\lambda}$.
	
	\item
	Suppose that $k\geq \frac{r+s}{2}$ and $k\geq s$. In this case, we have the following inequalities, $k\geq s\geq \frac{r+s}{2}\geq r$, $s\geq d\geq k-r+1$, and $k\leq r+s-1$. Let us proceed with the following consideration:
	$$
		P(\lambda)-P^{(4)}(\lambda)=4(s-d)Q(\lambda),
	$$
	where 
	$$
	Q(\lambda)=a_1\lambda+a_0,
	$$
	with
	$$
	a_0=-4(d-1)(k-d)(r-k+d-1),
	$$
	and
	$$
	a_1=(r-k+d-1)(s+r)+(r+s-k)(s-r).
	$$
	We observe that if $s=d$, then $P=P^{(4)}$. In particular, if $r=1$ or $k=d$, the equality $s=d$ holds and $P=P^{(4)}$. Now, we assume that $s>d$ and $r\ge2$. For $s=r$ and $d=k-r+1$, we obtain $P=P^{(4)}$.
	
	We only need to examine the remaining cases.
	
	Let $s>r$ and $d=k-r+1$, thus
	$$
	P(\lambda)-P^{(4)}(\lambda)=4(s-d)(r+s-k)(s-r)\lambda > 0,
	$$
	for any $\lambda >0.$
	
	For $s\geq r$ and $d>k-r+1$, then $Q(\lambda)$ is a line with positive slope since the inequality $s>d>k-r+1$ implies $r+s-k\geq 3$. Finally, we get, by $\tilde{\lambda}=\frac{rs}{2}$, 
	\begin{align*}
		Q(\tilde{\lambda})&\geq (d-(k-r+1))\left[(s+r)\frac{rs}{2}-4(d-1)(k-d)\right]\\
		&\geq (d-(k-r+1))\left[(s+r)\frac{rs}{2}-(k-1)^2\right]>0, 
	\end{align*}
	indeed as $r\geq 2$ and $s\geq 2$, then $rs-2s-2r+8>0$ or equivalently $\frac{rs}{2}>s+r-4\geq k-1$. 
	\end{enumerate}

\end{proof}
We are ready to present and prove the main result of this section.
\begin{theorem}\label{thm: complete bipartite graphs negative edges induce a tree}
	 If $m < \frac{r}{2}$ and $(K_{r,s},T^{*-}_m)$ maximizes the index among all signed graphs in $\mathcal{KT}_{r,s,m}$, then $T^*_m$ is isomorphic to $D_{k,1}$.
\end{theorem}
\begin{proof}
	 Note that $d_{T_m^*}(v_i)+d_{T_m^*}(v_j)\le m < \frac{r}{2}$ for $1\le i,j \le r$ and $d_{T_m^*}(w_i)+d_{T_m^*}(w_j)\le m < \frac{r}{2} \le \frac{s}{2}$ for $1\le i,j \le s$. Suppose that $T_m^*$ is not a star, neither a double star. Consequently, the diameter of $T_m^*$ is at least four. Let $u,\,v\in V(T_m^*)$ such that $u$ is leaf and there exists a path with three edges linking $u$ and $v$. If $u, v \in X$, following the idea and notation presented in the proof of Proposition \ref{thm: chain}, we denote by $w_2$ the unique neighbor of $u$ in $T_m^*$ and $w_1$ the neighbor of $v$ in the path, then $u \in \{v_1, \dots , v_k\}$ and $v \in \{v_{k+1}, \dots , v_\ell\}$. Therefore, the result follows from Proposition \ref{thm: chain} and Lemma \ref{thm: polynomials}. On the other hand, if $u, v \in Y$, for symmetry we conlude the statement. 
	
\end{proof}
\section{Unbalanced signed complete bipartite graphs}\label{sec: Unbalanced signed complete bipartite graphs}

In this section, we address the problem of finding the graph that maximizes the spectral radius among all unbalanced signed complete bipartite graphs. Note that $r\ge 2$; otherwise, the graphs would be balanced.

Let $\dot{G}^*=(K_{r,s},\sigma^*)$ be a signed complete bipartite graph with a bipartition $\{X,Y\}$, where $X=\{v_1,\ldots,v_r\}$ and $Y=\{w_1,\ldots,w_s\}$, $2\le r\le s$. The sign function $\sigma^*$ is defined such that $\sigma^*(xy)=-1$ if and only if $x=v_1$ and $y=w_1$ or vice versa. In this case, the induced subgraph of $\dot{G}^*$ consisting of the negative edges is isomorphic to $K_2$.

We begin by proving a technical lemma.

\begin{lemma}\label{lem: maximum nonnegative eigenvalue}
	 There exists an eigenvector $x\ge 0$ of $A(\dot{G}^*)$ associated with $\rho(\dot{G}^*)$. In addition, if $r>2$, then $x>0$; and, if $r=2$, $x$ has exactly one zero coordinate corresponding to the vertex $w_1$.
\end{lemma}

\begin{proof}
	 Let $x^t=(w^t,y^t)$ be an eignevector corresponding to $\rho(\dot{G}^*)$, where $w^t=(w_1,\ldots,w_r)$, $y^t=(y_1,\ldots,y_s)$. By Corollary \ref{cor: polinomios de estrellas}, $\rho(\dot{G}^*)=\sqrt{\frac{rs+\sqrt{z}}{2}}$ with $z=(rs)^2-16(r-1)(s-1)$. By symmetry, $w_2=w_i$ for each $2\le i\le r$, and $y_2=y_j$ for each $2\le j\le s$. It follows from the eigenvalue equations $A(\dot{G}^*)x=\rho(\dot{G}^*)x$ that
	\begin{eqnarray} 
		-w_1+(r-1)w_2 &=&  \rho(\dot{G}^*)y_1, \label{ec: eigeneq_1} \\
		w_1 + (r-1)w_2 &=&  \rho(\dot{G}^*)y_2, \label{ec: eigeneq_2} \\ 
		-y_1+(s-1)y_2 &=&  \rho(\dot{G}^*)w_1, \label{ec: eigeneq_3}\\ 
		y_1 + (s-1)y_2 &=&  \rho(\dot{G}^*)w_2. \label{ec: eigeneq_4} 
	\end{eqnarray}
	Hence, 
	\begin{eqnarray}
		(r-2)(s-1)y_2 &=& (\rho(\dot{G}^*)^2-r)y_1, \nonumber\\
		(r-2)y_1 &=& (\rho(\dot{G}^*)^2-r(s-1))y_2. \label{ec: eigeneq_5} 
 	\end{eqnarray}
	First, assume that $r>2$. Since $r<\rho(\dot{G}^*)^2$ and $r\le s$, $y_1=y_2=0$ or $y_1y_2>0$. Note the first case contradicting that $x$ is an eigenvector of $A(\dot{G}^*)$. Assume, without loosing generality, that $0<y_1$ and $0<y_2$. From \eqref{ec: eigeneq_4}, it follows that $w_2>0$. By equations \eqref{ec: eigeneq_1} and \eqref{ec: eigeneq_2},
	\[
	2w_1=\rho(\dot{G}^*)(y_2-y_1).
	\]
	We will prove that $y_1<y_2$. From \eqref{ec: eigeneq_5}, we obtain
	\[\frac{y_1}{y_2}=\frac{\rho(\dot{G}^*)^2-r(s-1)}{(r-2)}<1.\]
	Therefore, $w_1>0$ and $x>0$.
	
	Assume now that $r=2$. Since $r<\rho(\dot{G}^*)^2$, it follows that $y_1=0$. Besides, since $x\neq 0$, $\rho(\dot{G}^*)^2=2(s-1)$. Assume, without loss of generality, that $y_2=1$. From equations \eqref{ec: eigeneq_1} and \eqref{ec: eigeneq_2}, it follows that $w_1=w_2=\sqrt{\frac{s-1}{2}}$.
\end{proof}
In the following statement, we provide an upper bound on the spectral radius for any unbalanced signed complete bipartite graph.
\begin{theorem}\label{thm: complete bipartite}
	Let $\dot{G}=(K_{r,s}, \sigma)$ be an unbalanced signed complete bipartite graph. Then,
	\[\rho(\dot{G})\le \sqrt{\frac{rs+\sqrt{(rs)^2-16(r-1)(s-1)}}{2}}.\]
	Besides, the equality holds if and only if $\dot{G}$ is switching isomorphic to $\dot{G}^*$.
\end{theorem}
\begin{proof}
	By Corollary \ref{cor: polinomios de estrellas}, we have $\rho(\dot{G}^*)=\sqrt{\frac{rs+\sqrt{(rs)^2-16(r-1)(s-1)}}{2}}$. Let $\dot{G}_1=(K_{r,s},\sigma_1)$ be an unbalanced signed complete bipartite graph reaching the maximum spectrum radius among all unbalanced signed complete bipartite graphs. 
	
	As consequence of Lemma \ref{lem: spectral radius switching equivalent class}, there exists a signed graph $\dot{G}_2$ that is switching equivalent to $\dot{G}_1$ and a nonnegative eigenvector $x$ such that $A(\dot{G}_2)x=\rho(\dot{G}_1)x$, notice that $\rho(\dot{G}_2)=\rho(\dot{G}_1)$.  We can assume, without loss of generality, that $v_1w_1$ is a negative edge of $\dot{G}_2$. Hence,
	\[
	0\le\rho(\dot{G}_2)-\rho(\dot{G}^*)\le x^t\left(A(\dot{G}_2)-A(\dot{G}^*)\right)x\le 0,
	\]
	because clearly $A(\dot{G}_2)-A(\dot{G}^*)\le 0$. Therefore, we deduce that $\rho(\dot{G}_1)=\rho(\dot{G}^*)$ and $x^tA(\dot{G}^*)x=\rho(\dot{G}^*)$. In particular,  $x$ is an eigenvector of $\dot{G}^*$ associated with $\rho(\dot{G}^*)$. By Lemma \ref{lem: maximum nonnegative eigenvalue}, $x$ has all its coordinates positive, but $w_1$ whenever $r=2$. Therefore, since $x^t(A(\dot{G}_2)-A(\dot{G}^*))x=0$, we obtain that $A(\dot{G}_2)=A(\dot{G}^*)$, whenever $2<r$. When $r=2$ either, $A(\dot{G}_2)=A(\dot{G}^*)$, or $\dot{G}_2$ has $v_1w_1$ as its only negative edge, or  $\dot{G}_2$ has $v_1w_1$ and $v_2w_1$ as its only two negative edges. Note that the latter case is not possible because otherwise $\dot{G}_2$ would be balanced.
\end{proof}

The Lemma below follows by mimicking the proof of \cite[Lemma 4.5]{BrunettiandStanic2022}.

\begin{lemma}\label{lem: subgraphs chain}
	Let $\dot G$ be an unbalanced signed bipartite graph with a bipartition $\{U, W\}$ such that $|U|=n_1$ and $|W|=n_2$. Then $\dot G$ is subgraph of some unbalanced signed complete bipartite graph $\dot K_{n_1,n_2}$ such that $\lambda_1(\dot G)\le \lambda_1(\dot K_{n_1,n_2})$. Moreover, there exists a sequence $\{\dot G_i\}_{i=0}^{n_1n_2-m}$, where $\dot{G}_0=\dot{G}$, $m=|E(G)|$, and $\dot G_{n_1 n_2-m}=\dot K_{n_1,n_2}$. such that
	\begin{itemize}
		\item $\dot G_i$ is a proper subgraph of $\dot G_{i+1}$ for each $0 \le i\le n_1 n_2 - m - 1$, and
		\item $\lambda_1(\dot G_i)\le \lambda_1(\dot G_{i+1})$ for each $0 \le i\le n_1 n_2 - m - 1$.
	\end{itemize}	
\end{lemma}

Let us state and prove the main result of the article. To achieve this result, we need the following technical lemmas, each of which holds significance on its own.

\begin{lemma}\label{lem: complete bipartite without an edge and only one negative edge}
	Let $\dot G$ be an unbalanced signed bipartite graph with a bipartition $\{X, Y\}$ such that $|X|=r \le s=|Y|$ and $|E(\dot{G})|=rs-1$. If $\dot{G}$ has only one negative edge, then $\lambda_1(\dot{G})<\lambda_1(\dot{G}^*)$. 
\end{lemma}

\begin{proof}
	 We can assume, without loss of generality, that $v_1w_1$ is the negative edge of $\dot{G}$. To prove the result, we will analyze the square of the adjacency matrix of both signed graphs.  First consider the square of the adjacency matrix of $\dot{G}^*$,
	\[
        A(\dot{G}^*)^2=
        \begin{pmatrix}
		B^*&0\\
		0&C^*
        \end{pmatrix},
    \]
	where $B^*=\left(\begin{smallmatrix}
		s& (s-2)J_{1,r-1}\\
		(s-2)J_{r-1,1}& sJ_{r-1,r-1}
	\end{smallmatrix}\right)$. By Corollary \ref{cor: square spectrum}, $\rho(A(\dot{G}^*)^2)=\rho(B^*)$. The quotient matrix of $B^*$ is $Q^*=\left(\begin{smallmatrix}
	s& (r-1)(s-2)\\
	s-2& (r-1)s
\end{smallmatrix}\right)$ and its characteristic polynomial is $P^*(\lambda)=\lambda^2-rs\lambda+4(r-1)(s-1)$. 

We will split the proof into three cases.
\begin{enumerate} 

\item\label{lem: sin una arista - item 1} $\dot{G}=\dot{G}^*-\{v_1w_2\}$

Hence,
\[A(\dot{G})^2=\begin{pmatrix}
	B&0\\
	0&C
\end{pmatrix},\]
where $B=\left(\begin{smallmatrix}
	s-1           & (s-3)J_{1,r-1}\\
	(s-3)J{r-1,1} & s J_{r-1, r-1}
\end{smallmatrix}\right)$. Its quotient matrix, under the partition $\{\{v_1\},X\setminus\{v_1\}\}$, is $Q=\left(\begin{smallmatrix}
	s-1& \ \ (r-1)(s-3)\\
	s-3& \ \ (r-1)s
\end{smallmatrix}\right)$, and its characteristic polynomial is $P(\lambda)=\lambda^2-(rs-1)\lambda+(r-1)(5s-9)$. Besides, since $\dot{G}$ is unbalanced, $s\ge 3$.

We have $(P-P^*)(\lambda)=\lambda+(r-1)(s-5)>0$ for every $s\ge 5$ and $\lambda>0$. The remaining cases, i.e., $s\in\{3,4\}$, can be checked by computing the spectral radius of $\dot{G}$. Therefore, the result follows from lemmas \ref{lem: qotient matrix} and \ref{lem: comparing polynomials}.

\item $\dot{G}=\dot{G}^*-\{v_2w_1\}$

Hence,
\[A(\dot{G})^2=\begin{pmatrix}
	B & 0\\
	0&C
\end{pmatrix},\]
	where $B=\left(\begin{smallmatrix}
	s& s-1&(s-2) J_{1,r-2}\\
	s-1& s-1&(s-1)J_{1,r-2}\\
	(s-2)J_{r-2,1}& (s-1)J_{r-2,1}&sJ_{r-2,r-2}
\end{smallmatrix}\right)$. 

Its quotient matrix  is $Q=\left(\begin{smallmatrix}
s& \ \ s-1& \ \ (r-2)(s-2)\\
s-1& \ \ s-1& \ \ (r-2)(s-1)\\
s-2& \ \ s-1& \ \ (r-2)s
\end{smallmatrix}\right)$, under the partition $\{\{v_1\},\{v_2\},X\setminus\{v_1,v_2\}\}$, and its characteristic polynomial is $P(\lambda)=\lambda\left[\lambda^2-(rs-1)\lambda+(5r-9)(s-1)\right]$. Besides, since $\dot{G}$ is unbalanced, $r\ge 3$. 

We have $(P-\lambda P^*)(\lambda)=\lambda^2+(s-1)(r-5)\lambda>0$ for every $r\ge 5$ and $\lambda>0$. The remaining cases, i.e., $r\in\{3,4\}$, can be checked by computing the spectral radius of $\dot{G}$. Therefore, the result follows from lemmas \ref{lem: qotient matrix} and \ref{lem: comparing polynomials}.

\item $\dot{G}=\dot{G}^*-\{v_2w_2\}$

Hence,
\[A(\dot{G})^2=\begin{pmatrix}
	B & 0\\
	0&C
\end{pmatrix},\]
where $B=\left(\begin{smallmatrix}
	s & s-3 &(s-2) J_{1,r-2}\\
	s-3& s-1&(s-1) J_{1,r-2}\\
	(s-2) J_{r-2,1}& (s-1) J_{r-2,1} & s J_{r-2,r-2}
\end{smallmatrix}\right)$. 

The quotient matrix of $B$ is $Q=\left(\begin{smallmatrix}
	s& \ \ s-3& \ \ (r-2)(s-2)\\
	s-3& \ \ s-1& \ \ (r-2)(s-1)\\
	s-2& \ \ s-1& \ \ (r-2)s
\end{smallmatrix}\right)$, under the partition $\{\{v_1\},\{v_2\}, X\setminus\{v_1,v_2\}\}$, and its characteristic polynomial is $P(\lambda)=\lambda^3-(rs-1)\lambda^2+(5(r-1)(s-1)-4)\lambda-4(s-2)(r-2)$. We have $(P-\lambda P^*)(\lambda)=\lambda^2+(rs-r-s-3)\lambda-4(r-2)(s-2)$. On the one hand, if $r=2$, then $(P-\lambda P^*)(\lambda)=\lambda(\lambda+(r-1)(s-5))$, and the result follows from item \ref{lem: sin una arista - item 1}. On the other hand, if $r\ge 3$, then $\lambda^2-4(r-2)(s-2)>0$ when $\lambda\ge \rho(B)>\frac{rs} 2$. Hence, if $r\ge 3$,  $(P- P^*)(\lambda)\ge (P-\lambda P^*)(\lambda)>0$ when $\lambda\ge \rho(B)>\frac{rs} 2>1$. Therefore, the result follows from lemmas \ref{lem: qotient matrix} and \ref{lem: comparing polynomials}.

\end{enumerate}

\end{proof}
\begin{lemma}\label{lem: optimum presecribed partition}
	Let $\dot G$ be a connected unbalanced signed bipartite graph where $\{X,Y\}$ is the bipartition of $G$ with $2\le |X|=r\le s=|Y|$. Then, 
	\[\rho(\dot G)\le \sqrt{\frac{rs+\sqrt{(rs)^2-16(r-1)(s-1)}}{2}},\]
	and the equality holds if and only if $\dot G$ is switching isomorphic to $\dot{G}^*$. 
\end{lemma}

\begin{proof}
	Suppose that $|E(\dot{G})|< rs$, because otherwise the result holds (see Theorem \ref{thm: complete bipartite}). By Lemma \ref{lem: subgraphs chain}, there exists a signed graph $\dot H$ such that $|E(\dot{H})|=rs-1$, and
	\[\lambda_1(\dot G)\le \lambda_1(\dot H)\le \lambda_1(\dot{K}_{r,s}).\]
    If $\dot{K}_{r,s}$ is not switching isomorphic to $\dot{G}^*$, then by Theorem \ref{thm: complete bipartite}, 
	$\lambda_1(\dot G)\le \lambda_1(\dot H)\le\lambda_1(\dot{K}_{r,s})< \lambda_1(\dot{G}^*)$. Suppose that $\dot{K}_{r,s}$ is switching isomorphic to $\dot{G}^*$. Thus, there exists a diagonal matrix $D$ and a permutation matrix $P$ such that $A(\dot{G}^*)=P^{-1} D^{-1} A(\dot{K}_{r,s}) D \, P$. In addition, there exists an unbalanced signed graph $\dot{H'}$ switching isomorphic to $\dot{H}$ such that
	\[\lambda_1(\dot G)\le \lambda_1(\dot{H'})\le \lambda_1(\dot{G}^*),\]
	and with $v_1w_1$ as the only negative edge of $\dot{H'}$. By Lemma \ref{lem: complete bipartite without an edge and only one negative edge}, we have  
	\[\lambda_1(\dot G)\le \lambda_1(\dot{H'})< \lambda_1(\dot{G}^*).\]
\end{proof}
Now, we are ready to state the main result of this section, whose proof follows from Lemma \ref{lem: optimum presecribed partition}.
\begin{theorem}\label{thm: bipartite maximum}
	Let $\dot G$ be a connected unbalanced signed bipartite graph on $n$ vertices. If $r=\lfloor\frac n 2\rfloor$ and $s=\lceil\frac n 2\rceil$, then
	\[\rho(\dot G)\le \rho\left(\dot{G}^*\right),\]
	and the equality holds if and only if $\dot G$ is switching isomorphic to $\dot{G}^*$.
\end{theorem}
\section{Conclusions and further results}\label{sec: Conclusion and further results}
In Theorem \ref{thm: complete bipartite graphs negative edges induce a tree}, we identified the unbalanced signed complete bipartite graphs within the family $\mathcal {KT}_{r,s,m}$, with $r\le s$, that achieve the maximum spectral radius when $m < \frac{r}{2}$. In Theorem \ref{thm: bipartite maximum}, we established the list of unbalanced signed bipartite graphs on $n$ vertices that attain the maximum spectral radius. It is noteworthy that the problem of finding those graphs with minimum spectral radius  has ready been solved. Specifically, it has been proved that the unbalanced signed graphs on $n$ vertices with the minimum spectral radius are bipartite for $n\ge 4$ \cite[Theorem 3.5]{BrunettiandStanic2022} (see Fig. \ref{fig: bipartitos no balanceados minimos}).

\begin{theorem} 
For $n \ge 4$, up to switching isomorphism, the unbalanced  signed bipartite graphs minimizing the spectral radius are the following   
\begin{enumerate}
 \item $\dot{C}_4^{-}$ for $n=4$, 
 \item $\dot{\mathcal{Q}}_{0,1}$ for $n=5$,
 \item $\dot{C}_6^{-}$, $\dot{\mathcal{Q}}_{1,1}$ and $\dot{B}_{6}$ for $n=6$, 
 \item $\dot{B}_{7}$ for $n=7$,
 \item $\dot{U}_{1}$ for $n=8$,
 \item $\dot{\mathcal{Q}}_{r,s}$ for $n$ odd and $n \ge 9$, where $r=\lfloor\frac{n-4}{2}\rfloor$ and $s=\lceil\frac{n-4}{2}\rceil$,
 \item $\dot{\mathcal{Q}}_{r,r}$ or $\dot{C}_n^{-}$ for $n$ even and $n \ge 9$, where $r=\frac{n-4}{2}$.
\end{enumerate}
\end{theorem}

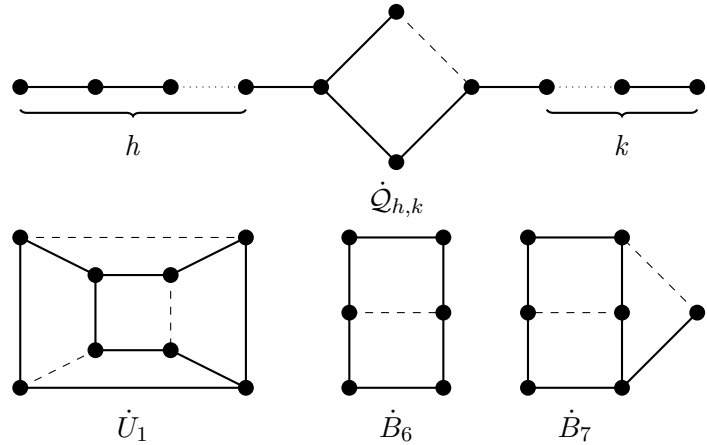
\begin{figure}[h!]
\begin{center}
	
	\begin{tikzpicture}[]
	
	\coordinate (w1) at (-5,0);
	\coordinate (w2) at (-4,0);
	\coordinate (w3) at (-3,0); 
	\coordinate (w4) at (-2,0); 
	\coordinate (w5) at (-1,0);
	\coordinate (w6) at (0,1);
	\coordinate (w7) at (0,-1); 
	\coordinate (w8) at (1,0); 
	\coordinate (w9) at (2,0); 
	\coordinate (w10) at (3,0);
	\coordinate (w11) at (4, 0);

	\foreach \punto in {1,...,11}
	\fill (w\punto) circle(3pt);
	
	\node[] at (0,-1.5) {$\dot{\mathcal{Q}}_{h,k}$};

	\draw[thick] (w1)--(w2);
	\draw[thick] (w2)--(w3);
	\draw[dotted] (w3)--(w4);
	\draw[thick] (w4)--(w5);
	\draw[thick] (w5)--(w6);
	\draw[thick] (w5)--(w7);
	\draw[dashed] (w6)--(w8);
	\draw[thick] (w7)--(w8);
	\draw[thick] (w8)--(w9);
	\draw[dotted] (w9)--(w10);
	\draw[thick] (w10)--(w11);
	
	\draw [thick, decoration={brace, raise=0.3cm, mirror}, decorate] 
    (w1) -- (w4) node [pos=0.5,anchor=north, yshift=-0.5cm] {\textit{h}};
    
    \draw [thick, decoration={brace, raise=0.3cm, mirror}, decorate] 
    (w9) -- (w11) node [pos=0.5,anchor=north, yshift=-0.5cm] {\textit{k}};
    
    \end{tikzpicture}
    
    \begin{tikzpicture}
    \coordinate (w1) at (-5,-4);
	\coordinate (w2) at (-2,-4);
	\coordinate (w3) at (-2,-6); 
	\coordinate (w4) at (-5,-6); 
	\coordinate (w5) at (-4,-4.5);
	\coordinate (w6) at (-3,-4.5);
	\coordinate (w7) at (-3,-5.5); 
	\coordinate (w8) at (-4,-5.5); 
		
	\foreach \punto in {1,...,8}
	\fill (w\punto) circle(3pt);

	\node[] at (-3.5,-6.5) {$\dot{U}_1$};

	\draw[dashed] (w1)--(w2);
	\draw[thick] (w2)--(w3);
	\draw[thick] (w3)--(w4);
	\draw[thick] (w4)--(w1);
	\draw[thick] (w5)--(w6);
	\draw[dashed] (w6)--(w7);
	\draw[thick] (w7)--(w8);
	\draw[thick] (w8)--(w5);
	\draw[thick] (w1)--(w5);
	\draw[thick] (w2)--(w6);
	\draw[thick] (w3)--(w7);
	\draw[dashed] (w4)--(w8);
    
    \coordinate (w1) at (-0.625,-4);
	\coordinate (w2) at (-0.625,-5); 
	\coordinate (w3) at (-0.625,-6); 
	\coordinate (w4) at (0.625,-6); 
	\coordinate (w5) at (0.625,-5);
	\coordinate (w6) at (0.625,-4);
		
	\foreach \punto in {1,...,6}
	\fill (w\punto) circle(3pt);

	\node[] at (-0,-6.5) {$\dot{B}_6$};

	\draw[thick] (w1)--(w2);
	\draw[thick] (w2)--(w3);
	\draw[thick] (w3)--(w4);
	\draw[thick] (w4)--(w5);
	\draw[thick] (w5)--(w6);
	\draw[thick] (w6)--(w1);
	\draw[dashed] (w2)--(w5);
    
    \coordinate (w1) at (3,-4);
	\coordinate (w2) at (3,-5); 
	\coordinate (w3) at (3,-6); 
	\coordinate (w4) at (1.75,-6); 
	\coordinate (w5) at (1.75,-5);
	\coordinate (w6) at (1.75,-4);
	\coordinate (w7) at (4,-5);
		
	\foreach \punto in {1,...,7}
	\fill (w\punto) circle(3pt);

	\node[] at (2.375,-6.5) {$\dot{B}_7$};

	\draw[thick] (w1)--(w2);
	\draw[thick] (w2)--(w3);
	\draw[thick] (w3)--(w4);
	\draw[thick] (w4)--(w5);
	\draw[thick] (w5)--(w6);
	\draw[thick] (w6)--(w1);
	\draw[dashed] (w2)--(w5);
    \draw[dashed] (w7)--(w1);
    \draw[thick] (w7)--(w3);
	\end{tikzpicture}
	
\end{center}

\caption{Unbalanced signed bipartite graphs with minimum spectral radius}\label{fig: bipartitos no balanceados minimos}

\end{figure}

\bibliographystyle{plain}
\bibliography{referencias}
\end{document}